\documentclass[reqno,11pt]{amsart}
\usepackage{amssymb,color,hyperref,mathrsfs,stmaryrd}
\usepackage{amsmath}
\usepackage{mathtools}

\usepackage[usenames,dvipsnames]{xcolor}
\usepackage[color,curve,matrix,arrow]{xy}

\numberwithin{equation}{section}
\newtheorem{theorem}{Theorem}

\newtheorem{corollary}{Corollary}

\theoremstyle{definition}
\newtheorem{definition}[equation]{Definition}

\newtheorem{rk}[equation]{Remark}

\newcommand{\F}{\mathcal{F}}

\newcommand{\Hom}{\operatorname{Hom}}
\newcommand{\Aut}{\operatorname{Aut}}

\newcommand{\Syl}{\operatorname{Syl}}

\newcommand{\ov}{\overline}
\def \<{\langle }
\def \>{\rangle }

\renewcommand{\phi}{\varphi}
\title[A characterization of saturated fusion systems over abelian $2$-groups]{A characterization of saturated fusion systems over abelian $2$-groups}
\author[E.~Henke]{Ellen Henke \vspace{0.2 cm}\\
\tiny Department of Mathematical Sciences, 
University of Copenhagen, Universitetsparken 5, 
DK-2100 Copenhagen, Denmark.  
Phone: 0045 50303418,
Email: henke@math.ku.dk
}
\thanks{The author was supported by the Danish National Research
  Foundation through the Centre for Symmetry and
  Deformation (DNRF92).}

\begin{document}

\maketitle

\begin{abstract}
Given a saturated fusion system $\F$ over a $2$-group $S$, we prove that $S$ is abelian provided any element of $S$ is $\F$-conjugate to an element of $Z(S)$. This generalizes a  Theorem of Camina--Herzog, leading to a significant simplification of its proof. More importantly, it follows that any  $2$-block $B$ of a finite group has abelian defect groups if  all $B$-subsections are major. Furthermore, every $2$-block with a symmetric stable center has abelian defect groups.
\end{abstract}

\bigskip

\textit{\textbf{Keywords:} Saturated fusion systems, modular representation theory, finite groups}

\section{Introduction}

This short note gives an example of how a conjecture in the modular representation theory of finite groups can be proved by showing its generalization to saturated fusion systems. We refer the reader to \cite{Aschbacher/Kessar/Oliver:2011a} for definitions and basic results regarding fusion systems and to \cite{NagaoTsushima1989} as a background reference on block theory.
Here is our main theorem:

\begin{theorem}\label{mainthm}
 Let $\F$ be a saturated fusion system on a $2$-group $S$ such that for any $x\in S$, $\Hom_\F(\<x\>,Z(S))\neq \emptyset$. Then $S$ is abelian.
\end{theorem}

%We remark that the corresponding statement is false if one replaces $2$ by an arbitrary prime. Counterexamples are the $3$-fusion systems of $Ru$, $J_4$ and $^2F_4(q)^\prime$ with $q=2^{6b\pm 1}$ for some non-negative integer $b$, the $5$-fusion system of $Th$, and the exotic $7$-fusion systems discovered by Ruiz and Viruel \cite{}.

%\smallskip

Since for any finite group $G$ with Sylow $2$-subgroup $S$ the fusion system $\F_S(G)$ is saturated, the above theorem yields immediately the following corollary:

\begin{corollary}[Camina--Herzog]\label{Cor1}
 Let $G$ be a finite group such that $|G:C_G(x)|$ is odd for any $2$-element $x$ of $G$. Then the Sylow $2$-subgroups of $G$ are abelian.
\end{corollary}

Corollary~\ref{Cor1} was first proved by Camina--Herzog \cite[Theorem~6]{Camina/Herzog1980}. As they point out, it means that one can read from the character table of a finite group if its Sylow $2$-subgroups are abelian. The proof of Camina--Herzog relies on the Theorem of Goldschmidt \cite{Goldschmidt:1974} about groups with strongly closed abelian $2$-subgroups, whereas our approach is elementary and self--contained. More precisely, we show Theorem~\ref{mainthm} by an  induction argument which appears natural in the context of fusion systems. Using the same idea, one can also give an elementary direct proof of Corollary~\ref{Cor1} which does not use fusion systems; see Remark~\ref{CaminaHerzog} for details.

\smallskip

We now turn attention to $2$-blocks of finite groups, where by a $p$-block we mean an  indecomposable direct summand of the group algebra $kG$ as a $(kG,kG)$-bimodule for an algebraically closed field $k$ of characteristic $p\neq 0$. The Brauer categories of $p$-blocks, which we introduce in Definition~\ref{BrauerCat}, provide important examples of saturated fusion systems. Given a finite group $G$ and a $p$-block $B$ of $G$, recall that a \textit{$B$-subsection} of $G$ is a pair $(x,b)$ such that $x$ is a $p$-element and $b$ is a block of $C_G(x)$ with the property that the induced block $b^G$ equals $B$. A $B$-subsection $(x,b)$ is called \textit{major} if the defect groups of $b$ are also defect groups of $B$. 
Theorem~\ref{mainthm} applied to the Brauer category of a $2$-block yields the following corollary, which we prove in detail at the end of this paper:

\begin{corollary}\label{Cor2}
 Suppose $B$ is a $2$-block of a finite group $G$ such that all $B$-subsections are major. Then the defect groups of $B$ are abelian.
\end{corollary}

Neither Theorem~\ref{mainthm} nor Corollary~\ref{Cor1} or Corollary~\ref{Cor2} have obvious generalizations replacing the prime $2$ by an arbitrary prime $p$. This is because for $p\in\{3,5\}$ there are finite groups $G$ in which the centralizer of any $p$-element has index prime to $p$ and the Sylow $p$-subgroups of $G$ are extraspecial of order $p^3$ and exponent $p$; see the examples listed in \cite[Theorem~A]{Kulshammer/Navarro/Sambale/Tiep}. The corresponding statement of Theorem~\ref{mainthm} fails also at the prime $7$, as the exotic $7$-fusion systems discovered by Ruiz and Viruel \cite{RuizViruel2004} show. 

\smallskip

We conclude with a further application of Theorem~\ref{mainthm}, concerning $2$-blocks of finite groups with a symmetric stable center. Recall the definitions of a symmetric algebra and the stable center of an algebra from \cite{Kessar/Linckelmann:2002}.
%Here the stable center of an algebra $A$ is a quotient of $Z(A)$ by a certain ideal; see \cite{Kessar/Linckelmann:2002} for the precise definition. Recall furthermore that a finite-dimensional $k$-algebra $A$ is defined to be symmetric if there exists a $k$-linear form $s:A\rightarrow k$ such that $s(ab)=s(ba)$ for any $a,b\in A$ and the kernel of $s$ contains no non-zero left or right ideal of $A$. 
Given a $p$-block $B$ of a finite group $G$ and a maximal $B$ subpair $(P,e)$, Kessar--Linckelmann \cite[Theorem~1.2]{Kessar/Linckelmann:2002} prove that the stable center of $B$ is symmetric provided the defect group $P$ of $B$ is abelian and $N_G(P,e)/C_G(P)$ acts freely on $P\backslash \{1\}$. Furthermore, by \cite[Theorem~1.1]{Kessar/Linckelmann:2002}, the converse implication is true if $B$ is the principal block of $G$. Our next corollary can be seen as a partial answer to the question in how far the converse holds for arbitrary $p$-blocks. This result is a direct consequence of Theorem~\ref{mainthm}, \cite[Theorem~3.1]{Kessar/Linckelmann:2002} and the fact that the Brauer category of a $p$-block is a saturated fusion system.

\begin{corollary}\label{Cor3}
 Let $B$ be a $2$-block of a finite group such that the stable center of $B$ is a symmetric algebra. Then $B$ has abelian defect groups.
\end{corollary}

\subsubsection*{Acknowledgement} Theorem~\ref{mainthm} and Corollary~\ref{Cor2} were conjectured by K\"uhlshammer--Navarro--Sambale--Tiep and answer the question posed in \cite[Question~3.1]{Kulshammer/Navarro/Sambale/Tiep} for $p=2$. The author would like to thank Benjamin Sambale for drawing her attention to this problem. Furthermore, she would like to thank an anonymous referee for suggesting Corollary~\ref{Cor3}.

\section{Proofs}

\begin{proof}[Proof of Theorem~\ref{mainthm}]
Let $\F$ be a counterexample to Theorem~\ref{mainthm} such that $|S|$ is minimal. 

\smallskip

\noindent{\em Step~1:} We show that $N_S(\<x\>)=C_S(x)$ for any $x\in S$.
By assumption, given $x\in S$, there exists $\phi\in\Hom_\F(\<x\>,S)$ such that $\phi(x)\in Z(S)$. In particular, $\phi(\<x\>)$ is fully normalized, so by \cite[I.2.6(c)]{Aschbacher/Kessar/Oliver:2011a}, there exists $\alpha\in\Hom_\F(N_S(\<x\>),S)$ such that $\alpha(\<x\>)=\phi(\<x\>)$. Then in particular, $\alpha(x)\in Z(S)$ and whence $[\alpha(x),\alpha(N_S(\<x\>))]=1$. As $\alpha$ is injective, this implies $[x,N_S(\<x\>)]=1$.

\smallskip

\noindent{\em Step~2:} We prove that $U:=\Omega_1(S)$ is elementary abelian. It is sufficient to show that any two involutions $u,v\in S$ commute. Note that $\<u,v\>$ is a dihedral subgroup of $S$, so there exists $x\in S$ with $\<u,v\>=\<x,u\>$ and $x^u=x^{-1}$. Then by Step~1, $[x,u]=1$, so $\<u,v\>$ is a four-group, which implies $[u,v]=1$.

\smallskip

\noindent{\em Step~3:} We now reach the final contradiction. By Step~2, $U$ is an elementary abelian subgroup of $S$, and since the image of an involution under an $\F$-morphism is again an involution, $U$ is strongly closed. Hence the factor system $\ov{\F}:=\F/U$ as in \cite[Section~II.5]{Aschbacher/Kessar/Oliver:2011a} is well-defined and by \cite[II.5.2, II.5.4]{Aschbacher/Kessar/Oliver:2011a} a saturated fusion system on $\ov{S}:=S/U$. By construction of $\ov{\F}$ and by \cite[I.4.7(a)]{Aschbacher/Kessar/Oliver:2011a} or \cite[II.5.10]{Aschbacher/Kessar/Oliver:2011a}, the morphisms of $\ov{\F}$ are precisely the group homomorphisms between subgroups of $\ov{S}$ which are induced by morphisms in $\F$. Hence, as $\Hom_\F(\<x\>,Z(S))\neq \emptyset$ for any $x\in S$, we have also 
\begin{equation}\label{star}
\tag{$*$}
\Hom_{\ov{\F}}(\<\ov{x}\>,\ov{Z(S)})\neq \emptyset \mbox{ for any }x\in S.
\end{equation}
In particular, since $\ov{Z(S)}\leq Z(\ov{S})$, the fusion system $\ov{\F}$ fulfills the assumptions of Theorem~\ref{mainthm}. So as $\F$ is a counterexample with $|S|$ minimal, $\ov{S}$ is abelian. Now by \cite[Theorem~3.6]{Aschbacher/Kessar/Oliver:2011a}, every morphism of $\ov{\F}$ extends to an element of $\Aut_{\ov{\F}}(\ov{S})$. By construction of $\ov{\F}$, the elements of $\Aut_{\ov{\F}}(\ov{S})$ are induced by elements of $\Aut_\F(S)$ and whence leave $\ov{Z(S)}$ invariant. It follows now from (\ref{star}) that $S=Z(S)U$. Hence, $S$ is abelian as $U$ is abelian, contradicting $\F$ being a counterexample.
\end{proof}

\begin{rk}[Proof of the Theorem of Camina--Herzog]\label{CaminaHerzog}
The proof of Corollary \ref{Cor1} in \cite{Camina/Herzog1980} also starts with showing that $U:=\Omega_1(S)$ is abelian for $S\in\Syl_p(G)$. An elementary proof of Corollary~\ref{Cor1} which does not rely on the theory of fusion systems, can be given by applying induction to $N_G(U)/U$ similarly as in Step~3 of our proof of Theorem~\ref{mainthm}. This requires to show that $N_G(U)$ controls fusion in $G$. In fact, by \cite[I.4.7(a)]{Aschbacher/Kessar/Oliver:2011a}, it is true in general that any abelian subgroup which is strongly closed in $S\in\Syl_p(G)$, controls fusion in $G$. However, in the special case we are in, there is a shorter argument to show that $N_G(U)$ controls fusion, which we give here:

\smallskip

We show first that $U\leq Z(S)$. Suppose by contradiction, there exists $a\in S$ with $[U,a]\neq 1$. Then we may choose $a$ of minimal order, which implies that $a$ acts as an involution on $U$. So $[U,a,a]=1$ and thus $U$ normalizes $W:=C_U(a)\<a\>$ as $U$ is abelian. By the assumption of the Theorem, there exists $g\in G$ with $a\in Z(S^g)$. As $W\leq C_G(a)$, by Sylow's Theorem we may assume $W\leq S^g$. Then $C_U(a)^{g^{-1}}\leq U$ as $U$ is strongly closed in $S$, yielding $C_U(a)\leq U^g$. Hence, since $U^g$ is abelian and $a\in Z(S^g)$, 
$$U^g\leq C_G(W)\unlhd N_G(W).$$ 
Note that for any $h\in G$, $U^h$ is strongly closed in $S^h$, so if $U\leq S^h$, then $U=U^h$ is strongly closed in $S^h$. In particular, $U$ is strongly closed in any $2$-subgroup containing $U$. Hence, as $U\leq N_G(W)$, it follows from Sylow's Theorem that $U$ is conjugate to $U^g$ by an element of $N_G(W)$ and thus $U\leq C_G(W)\leq C_G(a)$. This is a contradiction proving  $U\leq Z(S)$.

\smallskip

Let now $P\leq S$ and $x\in G$ such that $P^x\leq S$. By what we have just shown,  $U,U^x\leq C_G(P^x)$. Again, as $U$ is strongly closed in any $2$-subgroup containing $U$, there exists $c\in C_G(P^x)$ with $U^{xc}=U$. This proves that $N_G(U)$ controls fusion in $G$ as required.
\end{rk}

\begin{definition}[The Brauer category of a $p$-block]\label{BrauerCat}
 Let $G$ be a finite group, $p$ a prime, and $B$ a $p$-block of $G$. We refer the reader to \cite[Section~5.9.1]{NagaoTsushima1989} for the definitions of subpairs, $B$-subpairs, the relation $\unlhd$ and its transitive closure $\leq$, which is an ordering on subpairs. The defect groups of $B$ are precisely the $p$-subgroups $D$ of $G$ which occur as the first component of a maximal $B$-subpair of $G$. Fixing a maximal $B$-subpair $(D,b_D)$, for any subgroup $Q\leq D$, there exists a unique block $b_Q$ of $QC_G(Q)$ such that $(Q,b_Q)\leq (D,b_D)$. The Brauer category $\F_{(D,b_D)}(G,B)$ is the category whose objects are all subgroups of $D$ and, for $P,Q\leq D$, the set of morphisms from $P$ to $Q$ is given by 
$$\Hom_{\F_{(D,b_D)}(G,B)}(P,Q)=\{c_g:g\in G\mbox{ such that }(P^g,b_P^g)\leq (Q,b_Q)\},$$
where $c_g:P\rightarrow Q$ is defined via $c_g(x)=x^g$. It follows from \cite[Theorem~IV.3.2, Proposition~IV.3.14]{Aschbacher/Kessar/Oliver:2011a} that the category $\F_{(D,b_D)}(G,B)$ is a saturated fusion system on $D$. 
\end{definition}

\begin{proof}[Proof of Corollary~\ref{Cor2}]
Fix a maximal $B$-subpair $(D,b_D)$ of $G$ and set $\F:=\F_{(D,b_D)}(G,B)$. Let $x\in D$. Then by \cite[Theorem~5.9.3]{NagaoTsushima1989}, there exists a unique block $b_x$ of $C_G(x)$ such that $(\<x\>,b_x)\leq (D,b_D)$. Then $(x,b_x)$ is a $B$-subsection of $G$, which by assumption is major. Hence, by \cite[Theorem~5.9.6]{NagaoTsushima1989}, there exists $g\in G$ with $x^g\in Z(D)$ and $b_x^g=b_D^{C_G(x)}$. Note now the following general fact that follows from \cite[Lemma~5.3.4]{NagaoTsushima1989} and the definition of $\unlhd$: If $(P,b_P)$ is a $B$-subpair with $P\leq Z(D)$ and $b_P=b_D^{C_G(P)}$, then $(P,b_P)\unlhd (D,b_D)$. Hence, we have $(\<x^g\>,b_x^g)\unlhd (D,b_D)$ and thus $c_g:\<x\>\rightarrow D$ is a morphism in $\F$ by definition of the Brauer category. Therefore, $\F$ fulfills the assumption of Theorem~\ref{mainthm}, which implies that $D$ is abelian.  
\end{proof}

\bibliographystyle{amsplain}
\bibliography{repcoh}

\providecommand{\bysame}{\leavevmode\hbox to3em{\hrulefill}\thinspace}
\providecommand{\MR}{\relax\ifhmode\unskip\space\fi MR }
\providecommand{\MRhref}[2]{%
  \href{http://www.ams.org/mathscinet-getitem?mr=#1}{#2}
}
\providecommand{\href}[2]{#2}
\begin{thebibliography}{1}

\bibitem{Aschbacher/Kessar/Oliver:2011a}
M.~Aschbacher, R.~Kessar, and B.~Oliver, \emph{{Fusion systems in algebra and
  topology}}, London Math.\ Soc.\ Lecture Note Series, vol. 391, Cambridge
  University Press, 2011.

\bibitem{Camina/Herzog1980}
A.~R. Camina and M.~Herzog, \emph{{Character tables determine abelian Sylow
  $2$-subgroups}}, Proc. Amer. Math. Soc. \textbf{80} (1980), no.~3, 533--535.

\bibitem{Goldschmidt:1974}
D.~M. Goldschmidt, \emph{$2$-fusion in finite groups}, Ann. of Math. (2)
  \textbf{99} (1974), 70--117.

\bibitem{Kessar/Linckelmann:2002}
R.~Kessar and M.~Linckelmann, \emph{{On blocks with {F}robenius inertial
  quotient}}, J. Algebra \textbf{249} (2002), no.~1, 127--146.

\bibitem{Kulshammer/Navarro/Sambale/Tiep}
B.~K\"ulshammer, G.~Navarro, B.~Sambale, and P.~H. Tiep, \emph{{Finite groups
  with two conjugacy classes of $p$-elements and related questions for
  $p$-blocks}}, Preprint 2013.

\bibitem{NagaoTsushima1989}
H.~Nagao and Y.~Tsushima, \emph{{Representations of finite groups}}, Academic
  Press Inc., Boston, MA, 1989, Translated from the Japanese.

\bibitem{RuizViruel2004}
Albert Ruiz and Antonio Viruel, \emph{The classification of {$p$}-local finite
  groups over the extraspecial group of order {$p\sp 3$} and exponent {$p$}},
  Math. Z. \textbf{248} (2004), no.~1, 45--65.

\end{thebibliography}

\end{document}